\documentclass{amsart}
\usepackage{xypic,tikz}

\usepackage{amssymb,
amsmath,
amsthm,
graphicx,
amscd,
multind,
eufrak,
hyperref,
}
\usepackage[margin=1in]{geometry}
\theoremstyle{plain}
\newtheorem{thm}{Theorem}[section]
\newtheorem{lm}[thm]{Lemma}
\newtheorem{cor}[thm]{Corollary}
\newtheorem{prop}[thm]{Proposition}

\theoremstyle{definition}
\newtheorem{de}[thm]{Definition}

\newcommand{\CC}{{\mathbb C}}

\newcommand{\ZZz}{{\mathbb Z}_{\geq 0}}
\newcommand{\NN}{{\mathbb N}}
\newcommand{\PP}{{\mathbb P}}

{\begin{figure} \begin{center}}%
{\end{center} \end{figure}}
\newcommand{\Sym}{\operatorname{Sym}}
\newcommand{\rk}{\operatorname{rk}}

\newcommand{\GL}{\operatorname{GL}\nolimits}

\newcommand{\Ginf}{\GL_{\infty}}

\newcommand{\p}{p}
\newcommand{\q}{q}

\newcommand{\Mat}{{\operatorname{Mat}}}
\newcommand{\Mato}{\Mat_{\infty}}
\newcommand{\Mats}{{\Mato^{s}}}   
\newcommand{\ma}{M}                        

\newcommand{\MatfinoneN}[1][N]{{(\Mat_{{#1},{#1}})^{s}}}

\newcommand{\Symo}{X_{\infty}}
\newcommand{\Symp}{\Symo^{\p}}   
\newcommand{\Symfino}{{X_l}}
\newcommand{\Symfinone}{{X_l^{\p}}}
\newcommand{\SymfinoneN}[1][N]{{X_{#1}^{\p}}}

\newcommand{\Skew}{Y}

\newcommand{\Alto}{\Skew_{\infty}}
\newcommand{\Altq}{{\Alto^{\q}}}   
\newcommand{\Altfino}{{\Skew_l}}
\newcommand{\Altfinone}{{\Skew_l^{\q}}}
\newcommand{\AltfinoneN}[1][N]{{\Skew_{#1}^{\q}}}

\newcommand{\col}{c}                        

\newcommand{\Tot}{T}                        
\newcommand{\x}{x}                          
\newcommand{\xma}{x_{\mathrm{ma}}}          
\newcommand{\xsym}{x_{\mathrm{sym}}}          
\newcommand{\xalt}{x_{\mathrm{alt}}}          
\newcommand{\xcol}{x_{\mathrm{col}}}          
\newcommand{\xfin}{x_{\mathrm{fin}}}          

\newcommand{\Colo}{C_{\infty}}
\newcommand{\Coln}{\Colo^n}
\newcommand{\Colfino}{{C_l}}
\newcommand{\Colfinone}{{C_l^{n}}}
\newcommand{\ColfinoneN}[1][N]{{C_{#1}^{n}}}
\newcommand{\Cinf}{\Colo}

\newcommand{\Cinfw}{\CC^{\infty}}

\newcommand{\subsp}{Z}      

\begin{document}
\title[Finiteness properties of congruence classes of infinite matrices]{Finiteness properties of congruence classes of infinite matrices}
\author[R.H.~Eggermont]{Rob H. Eggermont}
\address[Rob H. Eggermont]{
Department of Mathematics and Computer Science\\
Technische Universiteit Eindhoven\\
P.O. Box 513, 5600 MB Eindhoven, The Netherlands}
\thanks{The author is supported by a Vidi grant from
the Netherlands Organisation for Scientific Research (NWO)}
\email{r.h.eggermont@tue.nl}

\date{Oktober 22, 2014}

\maketitle

\begin{abstract} We look at spaces of infinite-by-infinite matrices, and consider closed subsets that are stable under simultaneous row and column operations. We prove that up to symmetry, any of these closed subsets is defined by finitely many equations.
\end{abstract}

\section{Introduction}
\subsection{} Consider an infinite-dimensional affine space $X$ with an action by a large group $G$ of symmetries. Very little is known about finiteness properties of such spaces, and it is a big open problem to determine the extent of these properties. One such finiteness property is (set-theoretic) Noetherianity under the action of $G$, or simply $G$-Noetherianity, which is to say that $X$ does not contain any infinite strictly descending chain of $G$-stable closed subsets.

\subsection{} We consider the spaces $\Symo$ of infinite-by-infinite symmetric matrices and $\Alto$ of infinite-by-infinite skew-symmetric matrices. The group $\Ginf = \bigcup_{n\in\NN}\GL_n$ acts on these spaces by simultaneous row and column operations. That is to say, we have $g\cdot \ma = g\ma g^T$ for $g\in\Ginf$ and $\ma$ an element of $\Symo$ or of $\Alto$.

A (weak) version of our main theorem is the following.

\begin{thm}\label{thm:mainweak} The space $\Symp\times\Altq$ is $\Ginf$-Noetherian.
\end{thm}

\subsection{} The article \cite{DE14} shows that secants of Grassmannians are defined in bounded degree. One of the main auxiliary results in this article is that the space $\Mats$ of $s$-tuples of infinite-by-infinite matrices is set-theoretically Noetherian under the action of two copies of $\Ginf$, one acting by means of row operations, and one acting by means of column operations. This result follows from Theorem~\ref{thm:mainweak}. In fact, this theorem implies set-theoretical Noetherianity of $\Mats$ under \emph{simultaneous} row and column operations.

\subsection{} Finiteness properties have also been studied from a more algebraic point of view. The article \cite{SamSnow12} discusses the concept of a twisted commutative algebra, which can be seen as a large commutative ring $A$ with a large group action by $\Ginf$. Sam and Snowden describe the concept of weak Noetherianity, which says that any strictly ascending chain of $\Ginf$-stable ideals in $A$ must be of finite length. Results from \cite{abeasis1980ideali} and \cite{Abeasis1980158} can be used to show that the coordinate rings of $\Symo$ and $\Alto$ are weakly Noetherian. The spaces $\Symo$ and $\Alto$ are closely related to the representation theory of $\mathrm{O}_{\infty}$-modules and $\mathrm{Sp}_{\infty}$-modules, as seen in \cite{SamSnow13}.

Aside from $p+q \leq 1$, it is not known whether the coordinate ring of $\Symp\times\Altq$ is weakly Noetherian. Theorem~\ref{thm:mainweak} implies that any strictly ascending chain of $\Ginf$-stable \emph{radical} ideals in this coordinate ring must be of finite length. This is the strongest finiteness property known about this ring when $p+q > 1$.

\subsection{} Theorem~\ref{thm:mainweak} may seem surprising from an intuitive point of view. After all, for any $l \in \ZZz$, the dimension of $\Symfinone$, the space of $\p$-tuples of symmetric $l\times l$ matrices, is $\p\frac{l(l+1)}{2}$, while the dimension of $\GL_l$ is only $l^2$. So for $\p > 1$, one would expect the space $\Symp$ not to be Noetherian under $\Ginf$, as the dimension of the quotient space increases as $l$ grows. Indeed, if $(\mathrm{Id},\ma) \sim (\mathrm{Id},\ma')$ under $\Ginf$, then there must be an orthogonal matrix $O \in \Ginf$ such that $O\ma O^T =\ma'$, which is not the case generically. In fact, let $\ma, \ma'$ be diagonal matrices with entries $\lambda_1,\lambda_2,\ldots$ and $\lambda_1',\lambda_2',\ldots$. If at least one of the $\lambda_i$ is not equal to any of the $\lambda_j'$, then $(\mathrm{Id},\ma)$ and $(\mathrm{Id},\ma')$ have distinct orbits under $\Ginf$.

The above intuition is incorrect however. In the above example, if the diagonal matrix $\ma$ contains an infinite subset of pairwise distinct entries, then the $\Ginf$-orbit of $(\mathrm{Id},\ma)$ will in fact be dense in $\Symo^2$. In particular, even though the orbits of $(\mathrm{Id},\ma)$ and $(\mathrm{Id},\ma')$ are distinct, one can show that in this case, $(\mathrm{Id},\ma')$ is contained in the \emph{closure} of the orbit of $(\mathrm{Id},\ma)$.

The reason for this is that in the Zariski topology, each polynomial is only defined on a finite part of the matrix. This means that it suffices to show that for each $l$, there is $g\in\Ginf$ such that the first $l\times l$ blocks of $g\cdot (\mathrm{Id},\ma)$ equal the first $l\times l$ blocks of $(\mathrm{Id},\ma')$. The latter statement seems plausible if $\ma$ is sufficiently generic, and we will prove our main theorem by proving comparable statements.

\subsection{} One may wonder why we do not use the action $g\cdot M = gMg^{-1}$, rather than $g\cdot M = gMg^T$. One reason for this is that the former action does not preserve symmetric and skew-symmetric matrices, which are some of our objects of interest. In fact, knowledge of the (closures of the) orbits of symmetric and skew-symmetric matrices, and its relation to the rank of said matrices, is essential for our proofs. We don't have such knowledge for $g\cdot M = gMg^{-1}$, and there is more to the orbits than merely the rank of a matrix. For example, the orbit of the identity matrix consists of a single point. Furthermore, at some point in our proof, we work with limits of elements of $\Ginf$. While this is not a problem when working with $g\cdot M = gMg^T$, as this definition works for any matrix $g$ rather than just an invertible matrix, it can be a problem if we work with $g\cdot M = gMg^{-1}$.

The above reasons notwithstanding, we do not know whether a similar result to Theorem~\ref{thm:mainweak} holds if we work with $g\cdot M = gMg^{-1}$.

\subsection{} We hope that these results will eventually have applications similar to those mentioned in \cite{DE14}. Moreover, we hope that the method of proof might be of use in the study of symmetric and skew-symmetric $3$-tensors, for which no results are known regarding Noetherianity.

\section{Main theorem}
Let $\Cinfw = \bigcup_{n\in\NN}\CC^n$, and let $\Cinf = \{(c_n)_{n\in\NN}|\forall n: c_n\in\CC\}$. Note that we may view $\Cinf$ both as an infinite-dimensional affine space (with coordinate ring $\CC[x_i:i\in\NN]$) and as a vector space. We consider the space $\Mato = \{\ma = (m_{i,j})_{i,j\in\NN}: \forall i,j: m_{i,j}\in \CC\} \cong \mathrm{Hom}(\Cinfw,\Cinf)$.

Let $\Ginf = \bigcup_{n \in \NN} \GL(\CC^n)$. It acts on $\Cinf$ and $\Cinfw$ by left multiplication (viewing elements of these spaces as column matrices), and on $\Mato$ by $g\cdot \ma = g \ma g^T$.

We can decompose an element $\ma$ of $\Mato$ as $(\frac{\ma+\ma^T}{2},\frac{\ma-\ma^T}{2}) \in \Symo\times\Alto$, where $\Symo$ is the subspace of symmetric matrices in $\Mato$, and $\Alto$ is the subspace of skew-symmetric matrices in $\Mato$. The decomposition $\Mato = \Symo \times \Alto$ is canonical, and the action of $\Ginf$ respects this decomposition. Note that the spaces $\Mato,\Symo$, and $\Alto$ are all infinite-dimensional affine spaces (with coordinate rings $\Sym(\Cinfw\otimes\Cinfw),\Sym(\Sym^2(\Cinfw))$, and $\Sym(\bigwedge^2(\Cinfw))$ respectively).

We now come close to stating our main theorem. Let $X$ be a topological space, and let $G$ be a group acting on $X$. We say $X$ is $G$-Noetherian if every strictly descending chain of $G$-stable closed subsets has finite length.

\begin{thm}[Main Theorem]\label{thm:main} Let $\p,\q,n,m\in \ZZz$. Then the space $\Symp\times\Altq\times\Colo^n\times \CC^m$ is $\Ginf$-Noetherian.
\end{thm}

Note that when $\p = \q = 0$, the main theorem is true, and  \cite{Hillar09} shows a stronger result: The coordinate ring of the space $\Colo^n\times \CC^m$ is Noetherian under the action of the symmetric group, a much smaller group than $\Ginf$. We will not need this result in our proof though.

\begin{cor} For all $s \in \ZZz$, the space $\Mats$ is $\Ginf$-Noetherian.
\end{cor}

\begin{proof} The space $\Mats$ is isomorphic to $\Symo^s\times\Alto^s$ by means of a $\Ginf$-equivariant linear map.
\end{proof}

\begin{cor} Any strictly ascending chain of $\Ginf$-stable radical ideals in the coordinate ring of $\Symp\times\Altq\times\Coln\times \CC^m$ must be of finite length.
\end{cor}

\begin{proof} The corollary follows from the fact that any $\Ginf$-stable radical ideal in the coordinate ring of $\Symp\times\Altq\times\Coln\times \CC^m$ corresponds to a $\Ginf$-stable closed subspace of $\Symp\times\Altq\times\Coln\times \CC^m$ by means of an inclusion-reversing bijection.
\end{proof}

\begin{cor} Any $\Ginf$-stable closed subspace $\subsp$ of $\Symp\times\Altq\times\Coln\times \CC^m$ is cut out by the $\Ginf$-orbits of finitely many equations.
\end{cor}

\begin{proof} The ideal $I$ of $\subsp$ is defined by $\Ginf$-orbits of equations. Let $\subsp_0 = \Symp\times\Altq\times\Coln\times \CC^m$. Suppose we have a chain $\subsp_0\supsetneq \subsp_1\supsetneq \ldots \supsetneq \subsp_{k-1} \supseteq \subsp$, where $\subsp_i$ is cut out by the $\Ginf$-orbits of $i$ equations. If $\subsp_{k-1} \neq \subsp$, there is $f \in I$ that does not vanish on $\subsp_{k-1}$. Let $\subsp_k$ be the set of elements in $\subsp_{k-1}$ on which $\Ginf f$ vanishes. Then we have $\subsp_{k-1} \supsetneq \subsp_k \supseteq \subsp$, and $\subsp_k$ is cut out by the $\Ginf$-orbits of $k$ equations. By construction, each of the $\subsp_i$ is $\Ginf$-stable and closed, and since we cannot get an infinite strictly descending chain like this, we must have $\subsp_k = \subsp$ for some $k$.
\end{proof}

Note that this last corollary does not imply that the ideal of $\subsp$ is defined by the $\Ginf$-orbits of finitely many equations.

\section{Proof of the Main Theorem}
The rank of a linear map $V \to W$ is a well-defined element of $\ZZz\cup\{\infty\}$. It is invariant under the action of $\GL_V \times \GL_W$. We now introduce the concept of the rank of a tuple of linear maps.

\begin{de} Let $V,W$ be vector spaces, and let $\ma = (\ma_1,\ldots,\ma_s) \in \mathrm{Hom}(V,W)^s$. Then the \emph{rank} of $\ma$ is the infimum of $\rk(\sum_{i=1}^s \lambda_i\ma_i)$, where $\mathbf{\lambda} = (\lambda_1,\ldots,\lambda_s)$ runs over all non-zero elements of $\CC^s$.
\end{de}

We denote the rank of a tuple of linear maps $\ma$ by $\rk(\ma)$. Note that if $s > 0$, because $\ZZz \cup \{\infty\}$ is discrete, there is $(\lambda_1,\ldots,\lambda_s) \in \CC^s\setminus\{0\}$ such that $\rk(\sum_{i=1}^s \lambda_i\ma_i) = \rk(\ma)$. If $s=0$, we have $\rk(\ma) = \infty$ since we are taking the infimum of the empty set. It is easily seen that the rank of a tuple of linear maps is invariant under the action of $\GL_V\times \GL_W$. In particular, the rank of $\ma\in \Mats$ (viewing $\Mato$ as $\mathrm{Hom}(\Cinfw,\Cinf)$) is invariant under the action of $\Ginf$.

For convenience, we write $\Tot = \Tot_{\p,\q,n,m} = \Symp\times\Altq\times\Coln\times \CC^m$.  The following lemma will motivate our definition of rank of a matrix tuple.

\begin{lm}\label{lm:boundedrkNoeth} Let $(\p,\q,n,m) \in \ZZz^4$, and suppose the main theorem is true for all $(\p',\q',n',m')$ with $(\p',\q',n')$ lexicographically smaller than $(\p,\q,n)$. Then for all $r \in \ZZz$, the following sets are $\Ginf$-Noetherian:

\begin{itemize}
 \item $S_r = \{\x = (\xsym,\xalt,\xcol,\xfin) \in \Tot_{\p,\q,n,m}: \rk(\xsym) \leq r\}$;
 \item $\{\x \in \Tot: \rk(\xalt) \leq r\}$;
 \item $\{\x \in \Tot_{\p,\q,n,m}: \rk(\xcol) <n\}$.
\end{itemize}
\end{lm}

\begin{proof} We merely prove that $S_r$ is $\Ginf$-Noetherian; the other sets can be shown to be $\Ginf$-Noetherian by an analogous proof.

If $\p = 0$, we have $S_r = \emptyset$, which is $\Ginf$-Noetherian. Suppose $\p > 0$. We consider the map $\phi: \Tot_{\p-1,0,r,\p^2} \to \Tot_{\p,0,0,0}$ defined as follows. Let $\x \in \Tot_{\p-1,0,r,\p^2}$, and write $\x = (\xsym, \xcol,\xfin)$ with $\xsym \in \Symo^{\p-1}$, $\xcol \in \Colo^r$, and $\xfin = (\lambda_{i,j})_{i,j=1}^{\p} \in \CC^{\p^2}$. Now define $\phi(\x) \in \Symp$ by $\phi(\x)_i = \sum_{j=1}^{\p-1}\lambda_{i,j}(\xsym)_j + \lambda_{i,\p}\xcol(\xcol)^T$ for $i \in \{1,\ldots,\p\}$. Note that this map is $\Ginf$-equivariant, and its image consists of all elements in $\Symp$ of rank at most $r$. We extend $\phi$ to a map $\Tot_{\p-1,\q,n+r,m+\p^2} \to \Tot_{\p,\q,n,m}$, and observe that its image is $S_r$. Since $(\p-1,\q,n+r)$ is lexicographically smaller than $(\p,\q,n)$, the space $\Tot_{\p-1,\q,n+r,m+\p^2}$ is $\Ginf$-Noetherian, and hence so is $S_r$.
\end{proof}

The upshot of this lemma is that if $\subsp \subseteq \Tot$ is a closed $\Ginf$-stable subset that only contains elements $\x$ with $\rk(\xsym) < r$, $\rk(\xalt) < r$, or $\rk(\xcol) < n$, then it is $\Ginf$-Noetherian. In particular, if there would be an infinite strictly descending chain of closed $\Ginf$-stable subsets of $\Tot$, then for each element $\subsp$ of this chain and for all $r \in \ZZz$, there must be $\x \in \subsp$ such that $\rk(\xsym) \geq r$, $\rk(\xalt)\geq r$, and $\rk(\xcol) = n$.

For $l \in \ZZz$, let $\Symfino$, respectively $\Altfino$ be the space of symmetric, respectively skew-symmetric, $l\times l$ matrices, and let $\Colfino = \CC^l$. Our main proposition will be the following.

\begin{prop}\label{prop:highrank-denseorbit} Let $l \in \ZZz$. Then for any $\x = (\xsym,\xalt,\xcol) \in \Tot_{\p,\q,n,0} = \Symp\times \Altq\times \Coln$ with $\rk(\xsym) \gg 0$, $\rk(\xalt) \gg 0$, and $\rk(\xcol) = n$, the projection of $\Ginf \x$ to $\Symfinone\times \Altfinone \times \Colfinone$ is dominant.
\end{prop}

An immediate corollary of this proposition is that if $\subsp$ is any closed $\Ginf$-stable subset of $\Tot = \Tot_{\p,\q,n,0}$ such that $\subsp$ contains, for all $r\in\ZZz$, an element $\x$ with $\rk(\xsym)\geq r$, $\rk(\xalt)\geq r$, and $\rk(\xcol) = n$, then we have $\subsp = \Tot$. Indeed, any polynomial that vanishes on $\Tot$ is defined over some $\Symfinone\times \Altfinone \times \Colfinone$. Since $\subsp$ contains $\x$ such that $\Ginf\x$ projects dominantly to this set, such a polynomial must vanish identically on a dense subset of $\Symfinone\times \Altfinone \times \Colfinone$, and hence must be the zero polynomial. More importantly, this proposition also implies our main theorem.

\begin{proof}[Proof of Main Theorem given Proposition~\ref{prop:highrank-denseorbit}.] We apply induction to $(\p,\q,n)$ with lexicographic ordering. When $(\p,\q,n,m) = (0,0,0,m)$, the theorem is true, as this means $\Tot = K^m$. Now, fix $(\p,\q,n) \in \ZZz^3$, and assume the theorem is true for all $(\p',\q',n',m')$ with $(\p',\q',n')$ lexicographically smaller than $(\p,\q,n)$.

Let $\subsp_0 \supseteq \subsp_1 \supseteq \ldots$ be a descending chain of $\Ginf$-stable closed subsets of $\Tot$. We want to show that there exists $\subsp \subset \Tot$ such that $\subsp_i = \subsp$ for all $i \gg 0$. For $i \in \ZZz$, let $U_i \subseteq \subsp_i = \{\x \in \subsp_i: \rk(\xsym) > i, \rk(\xalt) > i, \rk(\xcol) = n\}$. Note that the $U_i$ are open and $\Ginf$-stable. We claim that it suffices to prove that there exists $U$ such that $\overline{U_r} = U$ for all $r \gg 0$. If this is the case namely, then for $r \gg 0$, the sets $V_r = \overline{\subsp_r \setminus \overline{U_r}}$ form a descending chain of closed, $\Ginf$-stable subsets in a space that is $\Ginf$-Noetherian by Lemma~\ref{lm:boundedrkNoeth}, and hence there is $V$ such that for $i \gg r$, we have $V_i = V$. Since $\subsp_i = \overline{U_i} \cup V_i$ for all $i$, we conclude that $\subsp_i = U \cup V$ for $i \gg 0$. So indeed, it suffices to prove that there exists $U$ such that $\overline{U_r} = U$ for all $r \gg 0$.

Let $\pi,\pi'$ be the projections from $\Tot_{\p,\q,n,m}$ to $\Tot_{\p,\q,n,0}$ and $K^m$ respectively. Since $K^m$ is Noetherian, there are $D \subseteq K^m$ and $R \in \ZZz$ such that for all $r \geq R$, we have $\overline{\pi'(U_r)} = D$.

Let $f$ be a polynomial defined on $\Tot_{\p,\q,n,m}$ that vanishes on $U_R$. Then $f$ is defined on $\Symfinone\times \Altfinone \times \Colfinone\times\CC^m$ for some $l$. We write $f = \sum f_i\otimes h_i$ with all $f_i$ defined on $\Symfinone\times \Altfinone \times \Colfinone$, all $h_i$ defined on $\CC^m$, and such that the $f_i$ are linearly independent. We proceed to show that all $h_i$ vanish on $D$.

Let $r \gg 0$. Then for any $\x \in U_r$, the projection of $\Ginf \pi(\x)$ to $\Symfinone\times \Altfinone \times \Colfinone$ is dominant by Proposition~\ref{prop:highrank-denseorbit}. Note that for any $g\in\Ginf$, we have $f(g\x) = \sum f_i(g\pi(\x))h_i(\pi'(\x))$. Because the $f_i$ are linearly independent and because $f(g\x) = 0$ for all $g\in\Ginf$ and $\x \in U_r$, we conclude $h_i(\pi'(\x)) = 0$ for all $i$ and all $\x\in U_r$. As the projection of $U_r$ is dense in $D$, we conclude all $h_i$ vanish on $D$, and hence $f$ vanishes on $\Symp\times \Altq\times \Colo^n\times D$. We conclude $\overline{U_r} = \Symp\times \Altq\times \Colo^n\times D$ for all $r\gg 0$. This concludes the proof.
\end{proof}

The main work will be proving Proposition~\ref{prop:highrank-denseorbit}. To do this, we start with two lemmas.

\begin{lm}\label{lm:reducefin}Let $V,W$ be vector spaces, let $\ma \in \mathrm{Hom}(V,W)^s$, and suppose $\ma$ has rank at least $r \in \ZZz$. Then there is a finite-dimensional subspace $U \subseteq V$ such that $\ma|_{\mathrm{Hom}(U,W)^s}$ has rank at least $r$.
\end{lm}

\begin{proof} If $V'$ is a finite-dimensional subspace of $V$, let $\pi_{V'}$ be the restriction $\mathrm{Hom}(V,W)^s \to \mathrm{Hom}(V',W)^s$, and let $D_{V'}= \{(d_1:\ldots:d_s)\in\PP^{s-1}: \rk(\pi_{V'}(\sum d_i\ma_i))<r\}$. Note that each $D_{V'}$ is closed. Suppose that $(d_1:\ldots:d_s) \in D_{V'}$. Then there is a $V' \subseteq V'' \subseteq V$ with $V''$ of finite dimension such that $(d_1:\ldots:d_s) \not\in D_{V''}$. In particular, this implies $D_{V'} \supsetneq D_{V''}$. Using Noetherianity of $\PP^{s-1}$, we conclude that there is a finite-dimensional subspace $U$ of $V$ such that $D_{U} = \emptyset$ (if not, we would be able to construct a strictly decreasing chain of closed subsets of $\PP^{s-1}$ of infinite length), which means $\pi_{U}(\ma)$ has rank at least $r$.
\end{proof}

\begin{lm}\label{lm:freevector}Let $V,W$ be vector spaces, let $\ma \in \mathrm{Hom}(V,W)^s$, and suppose $\ma$ has rank at least $s$. Then there is $v\in V$ such that $\ma v = \langle\ma_1v,\ldots,\ma_sv\rangle \subseteq W$ has dimension $s$.
\end{lm}

\begin{proof} By the previous lemma, we may assume $V$ is finite-dimensional without loss of generality. We let $\subsp := \{(v,d) \in V\times\PP^{s-1}:\sum d_i\ma_i v = 0\}$. For each $d \in \PP^{s-1}$, the codimension of the fiber above $d$ in $\subsp$ is at least $s$ since $\rk(\ma) \geq s$, and hence the codimension of $\subsp$ in $V\times\PP^{s-1}$ is at least $s$. This implies that the projection of $\subsp$ to $V$ has dimension at most $\mathrm{dim}(V)-1$, and hence there is $v\in V$ such that $\sum d_i\ma_i v \neq 0$ for any $d\in\PP^{s-1}$. This means $\ma v$ has dimension $s$, as was to be shown.
\end{proof}

An immediate corollary is the following.

\begin{cor}\label{cor:freeimage}Let $\ma \in \mathrm{Hom}(V,W)^s$, let $l \in \ZZz$, and suppose $\ma$ has rank at least $ls$. Then there is $V'\subseteq V$ of dimension $l$ such that $\ma V' = \ma_1V'+\ldots+\ma_sV'$ has dimension $ls$.
\end{cor}

\begin{proof} The corollary is clearly true for $l=0$, and the case $l=1$ is Lemma~\ref{lm:freevector}. Assume inductively that there is $V''\subseteq\Cinfw$ of dimension $l-1$ such that $\ma V''$ has dimension $(l-1)s$. Let $\overline{\ma}$ be the projection of $\ma$ to $\mathrm{Hom}(V/V'',W/\ma V'')^s$, and observe that $\overline{\ma}$ has rank at least $s$, since modding out $\ma V''$ reduces the rank of $\ma$ by $(l-1)s$, and since $V''$ simply maps to $0$ afterwards, modding out $V''$ does not reduce the rank any further. Now by Lemma~\ref{lm:freevector}, there is $v+V'' \in V/V''$ such that $\overline{\ma}\overline{v}$ has dimension $s$. Then clearly $V' = \langle v,V''\rangle$ is a $l$-dimensional subspace of $V$ such that $\ma V'$ has dimension $ls$, as was to be shown.
\end{proof}

Intuitively, the previous corollary allows us to play with the respective images of the $\ma_i$ without interfering with the other images. However, since $\Ginf$ acts on $\Mats$ by $g\cdot \ma_i = g\ma_i g^T$ rather than $g\cdot \ma_i = g\ma_i g^{-1}$, we are not able to freely choose bases, even if we restrict ourselves to $\GL_N$ acting on $\MatfinoneN$ for some $N$. We do have some freedom though, as shown in the following lemma.

\begin{lm}\label{lm:optionsG} Let $\ma \in \SymfinoneN \times \AltfinoneN$, let $s=\p+\q$, and suppose there is $V\subseteq \CC^N$ of dimension $2^sl$ such that $\ma V$ has dimension $s2^sl$. Then there is $g\in \GL_N$ such that $g\cdot\ma$ is of the form \[\left(
\begin{bmatrix}
0_l & *_{l,N-l} \\
\mathrm{Id}_l & *_{l,N-l} \\
0_l & *_{l,N-l} \\
\vdots & \vdots \\
0_l & *_{l,N-l} \\
0_{N-(s+1)l,l} & *_{N-(s+1)l,N-l}
\end{bmatrix},
\begin{bmatrix}
0_l & * \\
0_l & * \\
\mathrm{Id}_l & * \\
\vdots & \vdots \\
0_l & * \\
0_{N-(s+1)l,l} & *\end{bmatrix},
\ldots,
\begin{bmatrix}
0_l & * \\
0_l & * \\
\vdots & \vdots \\
\mathrm{Id}_l & * \\
0_{N-(s+1)l,l} & *\end{bmatrix}
\right).
\] Here, $*_{i,j}$ indicates an $i\times j$ matrix, $0_l$ indicates the $l\times l$ zero matrix, and $\mathrm{Id}_l$ indicates the $l\times l$ identity matrix.
\end{lm}

\begin{proof} We work in steps. First of all, we may assume $V = \langle e_1,\ldots,e_{2^sl}\rangle$ by applying some $h \in \GL_N$ such that $h^T \langle e_1,\ldots,e_{2^sl}\rangle = V$. The first $2^sl\times 2^sl$ blocks of $\ma$ are of the form $\ma'_i$ with $\ma'_i$ either symmetric or skew-symmetric.
If we now apply $h \in \GL_{2^sl} = \GL_V$, the space $(h\cdot\ma)V$ simply equals $h(\ma V)$, and therefore still has dimension $s2^sl$. With regards to the first $2^sl\times2^sl$ blocks of $\ma$, it acts by $(h\cdot \ma)'_i = g\ma'_ig^T$. Note that for all $r \in \{0,\ldots,2^sl\}$, there is a symmetric (respectively skew-symmetric) $2^sl\times 2^sl$ matrix of rank $r$ with the first $2^{s-1}l\times 2^{s-1}l$ block equal to $0$. Since any two symmetric (respectively skew-symmetric) matrices of the same rank lie in the same orbit under $\GL_V$, after applying some $h$, we may assume the first $2^{s-1}l\times 2^{s-1}l$ block of $\ma_s$ is $0$. Since any $h \in \GL_{2^{s-1}l}$ fixes this block, we can apply induction to $s$, and in doing so, we may assume that the first $l\times l$ block of each $\ma_i$ is $0_l$.

We have now reduced the problem to the case where each $\ma_i$ is of the form $\begin{pmatrix}0_l & *_{l,N-l}\\ *_{N-l,l} & *_{N-l,N-l}\end{pmatrix}$. Moreover, since $W = \langle e_1,\ldots,e_l\rangle \subseteq V$, we find $\ma W$ has dimension $sl$, which means the first $l$ columns of the $s$ matrices are linearly independent.

Let $W' = \langle e_{l+1},\ldots,e_N\rangle$. We apply $h \in \GL_{W'}$. Note that we have $h\ma_i = \\ \begin{pmatrix}\mathrm{Id}_l & 0_{l,N-l}\\ 0_{N-l,l} & h_{N-l,N-l}\end{pmatrix}\begin{pmatrix}0_l & *_{l,N-l}\\ *_{N-l,l} & *_{N-l,N-l}\end{pmatrix}\begin{pmatrix}\mathrm{Id}_l & 0_{l,N-l}\\ 0_{N-l,l} & h_{N-l,N-l}^T\end{pmatrix}$, and we observe that this changes the first $l$ columns of $\ma_i$ from $\begin{pmatrix}0_l\\ *_{N-l,l}\end{pmatrix}$ to $\begin{pmatrix}0_l\\ h_{N-l,N-l}*_{N-l,l}\end{pmatrix}$. In other words, we can apply arbitrary row operations to the first $l$ columns of $\ma_i$ (as long as we apply them to all $\ma_i$ simultaneously). Since the first $l$ columns of the $s$ matrices are all linearly independent, it is now easily verified that we may indeed assume that $\ma$ has the desired form after applying some $g \in \GL_N$.
\end{proof}

\begin{cor}\label{cor:closureG} Let $\ma \in \SymfinoneN \times \AltfinoneN$, let $s=\p+\q$, and suppose there is $V\subseteq \CC^N$ of dimension $2^sl$ such that $\ma V$ has dimension $s2^sl$. Then the closure of $\GL_N \ma$ contains \[\left(
\begin{bmatrix}
0_l & *_{l,N-l} \\
\mathrm{Id}_l & 0_{l,N-l} \\
0_l & 0_{l,N-l} \\
\vdots & \vdots \\
0_l & 0_{l,N-l} \\
0_{N-(s+1)l,l} & 0_{N-(s+1)l,N-l}
\end{bmatrix},
\begin{bmatrix}
0_l & * \\
0_l & 0 \\
\mathrm{Id}_l & 0 \\
\vdots & \vdots \\
0_l & 0 \\
0_{N-(s+1)l,l} & 0\end{bmatrix},
\ldots,
\begin{bmatrix}
0_l & * \\
0_l & 0 \\
\vdots & \vdots \\
\mathrm{Id}_l & 0 \\
0_{N-(s+1)l,l} & 0\end{bmatrix}
\right).
\]
\end{cor}

\begin{proof} By Lemma~\ref{lm:optionsG}, we may assume the first $l$ rows and columns of each $\ma_i$ have the desired form. Multiply $\ma$ with $\begin{bmatrix} \frac{1}{\lambda} \mathrm{Id}_l & 0\\0 & \lambda\mathrm{Id}_{N-l}\end{bmatrix}$ for $\lambda \neq 0$. It maps $\begin{bmatrix} 0_l & B_{l,N-l}\\C_{N-l,l} & D_{N-l,N-l}\end{bmatrix}$ to $\begin{bmatrix} 0_l & B_{l,N-l}\\C_{N-l,l} & \lambda^2D_{N-l,N-l}\end{bmatrix}$. Now let $\lambda$ go to $0$.
\end{proof}

We are now at the point where we can prove Proposition~\ref{prop:highrank-denseorbit}. Since the proof becomes rather technical, we first prove a weaker version of the proposition. We only prove this lemma in order to get a feeling for what is happening. We will not use it in the actual proof of Proposition~\ref{prop:highrank-denseorbit}.

\begin{lm}\label{lm:highrank-denseorbit} Let $l \in \ZZz$. Then for any $\ma \in \Symp$ such that $\rk(\ma) \geq \p2^{\p}l$, the projection of $\Ginf \ma$ to $\Symfinone$ is dominant.
\end{lm}

\begin{proof} Let $\ma \in \Symp$ such that $\rk(\ma)\geq 2^{\p}l\p$. By Corollary~\ref{cor:freeimage}, there is $V' \subseteq \Cinfw$ of dimension $2^{\p}l$ such that $\ma V'$ has dimension $\p2^{\p}l$. Note that it suffices to show that the $\Ginf$-orbit of something in the closure of $\Ginf \ma$ projects dominantly to $\Symfinone$.

By Corollary~\ref{cor:closureG}, we may assume that (after projecting down to some $\CC^N$ with $N$ sufficiently large and taking an element in the closure), we have \[\ma = \left(
\begin{bmatrix}
0_l & *_{l,N-l} \\
\mathrm{Id}_l & 0_{l,N-l} \\
0_l & 0_{l,N-l} \\
\vdots & \vdots \\
0_l & 0_{l,N-l} \\
0_{N-(\p+1)l,l} & 0_{N-(\p+1)l,N-l}
\end{bmatrix},
\begin{bmatrix}
0_l & * \\
0_l & 0 \\
\mathrm{Id}_l & 0 \\
\vdots & \vdots \\
0_l & 0 \\
0_{N-(\p+1)l,l} & 0\end{bmatrix},
\ldots,
\begin{bmatrix}
0_l & * \\
0_l & 0 \\
0_l & 0 \\
\vdots & \vdots \\
\mathrm{Id}_l & 0 \\
0_{N-(\p+1)l,l} & 0\end{bmatrix}
\right).\]

We project down to $\CC^{(\p+1)l}$, and restrict ourselves to the action of $\GL_{(\p+1)l}$. We can now write $\ma_i = (\ma_i^{j,k})_{j,k=1}^{\p+1}$ where each $\ma_i^{j,k}$ is an $l\times l$ matrix. By the choices we made, we have $\ma_i^{1,1} = 0$ for all $i$, and we have $\ma_i^{j+1,1} = \delta_{i,j}\cdot \mathrm{Id}_l$ for all $i,j \in \{1,\ldots,\p\}$. Finally, we have $\ma_i^{j,k} = 0$ for $j,k \geq 2$. Since all $\ma_i$ are symmetric, we also have $\ma_i^{1,j+1} = \delta_{i,j}\cdot \mathrm{Id}_l$ for all $i,j$.

Let $\ma'_1,\ldots,\ma'_{\p} \in \Symfinone$.  Let \[g = \mathrm{Id}+ \begin{pmatrix}0_l & \frac{1}{2}\ma'_1 & \frac{1}{2}\ma'_2 & \ldots & \frac{1}{2}\ma'_{\p}\\ \\ & & 0_{\p l,(\p+1)l} \end{pmatrix}.\]

By direct computation, we find that the first $l\times l$ block of $g\ma_ig^T$ is $\frac{1}{2}\ma'_i+\frac{1}{2}(\ma'_i)^T$, which equals $\ma'_i$ because $\ma'_i$ is symmetric. We conclude that $\ma'$ lies in the closure of the projection of $\GL_{(\p+1)l}\ma$ to $\Symfinone$. This concludes the proof.
\end{proof}

An analogous proof can be used for $\Altq$, and for $\Symp\times\Altq$. We now prove our main proposition.

\begin{proof}[Proof of Proposition~\ref{prop:highrank-denseorbit}]
Write $s = \p+\q$, and let $r = s2^sl+2(s+1)n$. Let $\x = (\xsym,\xalt,\xcol) \in \Symp\times\Altq\times\Colo^n$ with $\rk(\xsym),\rk(\xalt) \geq 2r$, and $\rk(\xcol) = n$. Observe that $\xma = (\xsym,\xalt)$ has rank at least $r$. If not namely, let $\lambda_i \in \CC$ such that $M = \sum \lambda_i (\xma)_i$ has rank smaller than $r$. Then both $M+M^T$ and $M-M^T$ have rank smaller than $2r$. However, since at least one $\lambda_i$ is non-zero, at least one of these is a non-trivial linear combination of the $\xsym$ or $\xalt$, and hence should have rank at least $2r$, a contradiction.

Project $\x$ to $\SymfinoneN\times\AltfinoneN\times\ColfinoneN$ with $N$ large enough such that the projection of $\x$ (which we also denote by $\x$ for convenience) still satisfies $\rk(\xsym),\rk(\xalt) \geq 2r$ and $\rk(\xcol) = n$. Write $\xma = (\xsym,\xalt)$; it has rank at least $r$. By applying some $g \in \Ginf$, we may assume the span of $\xcol$, and $(\xma)_i\xcol$ for $i \in \{1,\ldots,s\}$ is contained in $\langle e_{N-(s+1)n+1,\ldots,e_N}$. Moreover, we may assume $\xcol = \begin{pmatrix}0_{n,N-n} & \mathrm{Id}_n\end{pmatrix}$

In particular, any $(\xma)_i$ has the form $\begin{pmatrix} *_{N-(s+1)n,N-(s+1)n} & 0_{N-(s+1)n,(s+1)n}\\0_{(s+1)n,N-(s+1)n} & *_{(s+1)n,(s+1)n}\end{pmatrix}$. Here, we implicitly use the fact that any $(\xma)_i$ is either symmetric or skew-symmetric.

The projection $\overline{\xma}$ of $\xma$ to $\SymfinoneN[N-(s+1)n]\times \AltfinoneN[N-(s+1)n]\times \ColfinoneN[N-(s+1)n]$ has rank at least $r-2(s+1)n = s2^sl$ (because we remove $(s+1)n$ rows and columns), and hence there is $V \subseteq \CC^{N-(s+1)n}$ of dimension $2^sl$ such that $\overline{\xma} V$ has dimension $s2^sl$. But then the same is true for $\xma$, and we have $\xma V \subseteq \CC^{N-(s+1)n}$.

By Lemma~\ref{lm:optionsG}, and after permuting coordinates and projecting down to $\SymfinoneN[(s+1)l+n]\times \AltfinoneN[(s+1)l+n]\times \ColfinoneN[(s+1)l+n]$ we may now assume that we have \[
\xma=\left(
\begin{bmatrix}
0_l & *_{l,sl+n} \\
\mathrm{Id}_l & *_{l,sl+n} \\
0_l & *_{l,sl+n} \\
\vdots & \vdots \\
0_l & *_{l,sl+n} \\
0_{n,l} & *_{n,n}
\end{bmatrix},
\begin{bmatrix}
0_l & * \\
0_l & * \\
\mathrm{Id}_l & * \\
\vdots & \vdots \\
0_l & * \\
0_{n,l} & *\end{bmatrix},
\ldots,
\begin{bmatrix}
0_l & * \\
0_l & * \\
0_l & * \\
\vdots & \vdots \\
\mathrm{Id}_l & * \\
0_{n,l} & *\end{bmatrix}
\right),
\] that $\xcol = \begin{pmatrix}0_{n,(s+1)l},\mathrm{Id}_n\end{pmatrix}$, and that the last $n$ rows of each $(\xma)_i$ are of the form $\begin{pmatrix}0_{n,(s+1)l} & *_{n,n}\end{pmatrix}$. Since each of the $(\xma)_i$ is symmetric or skew-symmetric, we have similar properties for the first $l$ rows and the last $n$ columns of each $(\xma)_i$. Using a proof similar to the proof of Corollary~\ref{cor:closureG}, we may assume without loss of generality that the $j,k$-th entry of $(\xma)_i$ is zero for all $j,k \in \{l+1,\ldots,(s+1)l\}$.

Let $\ma' \in \Symfinone\times\Altfinone$, and let $\col' \in \Colfinone$. Let $\tilde{\ma}_i$ be the final $n\times n$ block of $(\xma)_i$, and write $\ma''_i = \col'\tilde{\ma}_i(\col')^T$. We now define \[g = \mathrm{Id} + \begin{pmatrix}0_l & \frac{1}{2}(\ma'_1-\ma''_1) & \frac{1}{2}(\ma'_2-\ma''_2) & \ldots & \frac{1}{2}(\ma'_s-\ma''_s) & \col'\\ \\ & & 0_{n-l,n}\end{pmatrix}.\] One now verifies by direct computation that the first $n\times l$ block of $g\xcol$ equals $\col'$, and that the first $l\times l$ block of $g(\xma)_ig^T$ equals $\ma'_i$. This concludes the proof.
\end{proof}

Note that we do not prove that the bound $r \geq s2^sl+2(s+1)n$ is sharp. In fact, one can use simple arguments to reduce the bound to $r\geq s2^sl+(s+1)n$. Even then, in the case $s=1, n=0$ for example, one can easily see that to get a dense orbit in the space of $l \times l$ symmetric or skew-symmetric matrices, it suffices to take $r\geq l$ rather than $r \geq 2l$.


\begin{thebibliography}{Abe80}

\bibitem[Abe80]{abeasis1980ideali}
S.~Abeasis.
\newblock Gli ideali {GL}({V})-invarianti in $\mathrm{S}$($\mathrm{S}^2$({V})).
\newblock {\em Rendiconti Math}, 13:235--262, 1980.

\bibitem[AF80]{Abeasis1980158}
S.~Abeasis and A.~Del Fra.
\newblock Young diagrams and ideals of pfaffians.
\newblock {\em Advances in Mathematics}, 35(2):158 -- 178, 1980.

\bibitem[DE14]{DE14}
J.~{Draisma} and R.~H. {Eggermont}.
\newblock {Pl\"{u}cker varieties and higher secants of Sato's Grassmannian}.
\newblock {\em ArXiv e-prints}, February 2014.
\newblock Available from \verb+http://arxiv.org/abs/1402.1667+.

\bibitem[HS09]{Hillar09}
Chris~J. Hillar and Seth Sullivant.
\newblock Finite {G}r\"obner bases in infinite dimensional polynomial rings and
  applications.
\newblock Preprint, available from \verb+http://arxiv.org/abs/0908.1777+, 2009.

\bibitem[SS12]{SamSnow12}
S.~V. {Sam} and A.~{Snowden}.
\newblock {Introduction to twisted commutative algebras}.
\newblock {\em ArXiv e-prints}, September 2012.
\newblock Available from \verb+http://arxiv.org/abs/1402.1667+.

\bibitem[SS13]{SamSnow13}
S.~V. {Sam} and A.~{Snowden}.
\newblock {Stability patterns in representation theory}.
\newblock {\em ArXiv e-prints}, February 2013.
\newblock Available from \verb+http://arxiv.org/abs/1302.5859+.

\end{thebibliography}
\end{document}